\documentclass[11pt]{article}
\usepackage[english]{babel}
\usepackage{amssymb,amsmath,amsthm}
\newtheorem{theorem}{Theorem}[section]
\newtheorem{lemma}[theorem]{Lemma}

\newtheorem{corollary}[theorem]{Corollary}

\newtheorem{question}[theorem]{Question}

{\theoremstyle{definition}}
{\theoremstyle{definition}}
{\theoremstyle{definition}}
{\theoremstyle{definition}\newtheorem{remark}[theorem]{Remark}}

\newtheorem*{cons}{Question S}
\newtheorem*{thmb}{Theorem B}
\newtheorem*{thmbg}{Theorem BG}

\numberwithin{equation}{section}

\def\C{{\mathbb C}}
\def\N{{\mathbb N}}
\def\Z{{\mathbb Z}}
\def\R{{\mathbb R}}
\def\T{{\mathbb T}}

\def\H{{\mathcal H}}

\def\K{{\mathcal K}}

\def\epsilon{\varepsilon}
\def\kappa{\varkappa}
\def\phi{\varphi}
\def\leq{\leqslant}
\def\geq{\geqslant}

\def\dim{{\rm dim}\,}
\def\ssub#1#2{#1_{{}_{{\scriptstyle #2}}}}
\def\ker{\hbox{\tt ker}\,}
\def\spann{\hbox{\tt span}\,}
\def\dist{\hbox{\tt dist}\,}
\def\spannn{\overline{\hbox{\tt span}}\,}

\def\PLA{\vrule width0pt height14pt depth10pt}
\def\PLAN{\vrule width0pt height16pt depth12pt}
\def\ssub#1#2{#1_{{}_{{\scriptstyle #2}}}}

\title{A hypercyclic finite rank perturbation of a unitary operator}

\author{Stanislav Shkarin}

\date{}

\begin{document}

\maketitle

\begin{abstract} A unitary operator $V$ and a rank $2$
operator $R$ acting on a Hilbert space $\H$ are constructed such
that $V+R$ is hypercyclic. This answers affirmatively a question of
Salas whether a finite rank perturbation of a hyponormal operator
can be supercyclic.
\end{abstract}



\section{Introduction \label{s1}}\rm

{\it All vector spaces in this article are assumed to be over the
field} \ $\C$ of complex numbers. Symbol $\R$ stands for the field
of real numbers, $\Z_+$ is the set of non-negative integers, $\N$ is
the set of positive integers and $\T=\{z\in\C:|z|=1\}$. For a subset
$A$ of a Banach space $X$, $\spann(A)$ stands for the linear span of
$A$ and $\spannn(A)$ denotes the closure of $\spann(A)$. For a
Banach space $X$, $L(X)$ is the Banach algebra of continuous linear
operators on $X$. In what follows symbol $\mu$ stands for the
normalized Lebesgue measure on $\T$. Instead of $L_1(\T,\mu)$ and
$L_2(\T,\mu)$ we simply write $L_1(\T)$ and $L_2(\T)$. We use the
$\langle x,y\rangle$ notation for the scalar product of vectors $x$
and $y$ in a Hilbert space $\H$. A compact subset of a metric space
is called {\it perfect} if it is non-empty and has no isolated
points.

Recall that a continuous linear operator $T$ on a topological vector
space $X$ is called {\it hypercyclic} if there exists $x\in X$ such
that the orbit $\{T^nx:n\in\Z_+\}$ is dense in $X$ and $T$ is called
{\it supercyclic} if there is $x\in X$ for which the projective
orbit $\{\lambda T^nx:\lambda\in\C,\ n\in\Z_+\}$ is dense in $X$. We
refer to \cite{bama} for additional information on hypercyclicity
and supercyclicity. In particular, it is well  known that there are
no hypercyclic operators on finite dimensional Hausdorff topological
vector spaces and there are no supercyclic operators on Hausdorff
topological vector spaces of finite dimension $\geq2$. Thus when
speaking of hypercyclicity or supercyclicity of operators, we {\it
always assume that the underlying space is infinite dimensional}.

Recall also that a bounded linear operator $T$ on a Hilbert space
$\H$ is called {\it hyponormal} if $\|Tx\|\geq \|T^*x\|$ for any
$x\in\H$, where $T^*$ is the Hilbert space adjoint of $T$.
Equivalently, $T$ is hyponormal if and only if $T^*T-TT^*\geq 0$.

Hilden and Wallen \cite{hw} observed that there are no supercyclic
normal operators. Kitai \cite{ansa} proved that there are no
hypercyclic hyponormal operators. A result simultaneously stronger
than those of Hilden and Wallen and of Kitai was obtained by Bourdon
\cite{bdon}, who demonstrated that a hyponormal operator can not be
supercyclic. This motivated Salas \cite{sal1} to raise the following
question.

\begin{cons} Can a finite rank perturbation of a hyponormal operator
be supercyclic?
\end{cons}

The above question is also reproduced in \cite{moga}. It is worth
noting that Bayart and Matheron \cite{bm} constructed a unitary
operator on a Hilbert space $\H$, which is supercyclic on $\H$
endowed with the weak topology (=weakly supercyclic). In the present
paper we answer Question~S affirmatively.

\begin{theorem}\label{main} There exist a unitary operator $V$ and a
bounded linear operator $R$ of rank at most $2$ acting on a Hilbert
space $\H$ such that $T=V+R$ is hypercyclic.
\end{theorem}

The idea of the proof is the following. We consider the unitary
multiplication operator $U$ on $L_2(\T)$, $Uf(z)=zf(z)$, and
construct $h,g\in L_2(\T)$ and a closed linear subspace $\K$ of
$L_2(\T)$ such that $\K$ is invariant for $U+S$, where $Sf=\langle
f,g\rangle h$, the restriction $T\in L(\K)$ of $U+S$ to $\K$ is
hypercyclic and $T$ admits the decomposition $T=V+R$, with $V\in
L(\K)$ being unitary and $R\in L(\K)$ having rank at most $2$. We
prove the hypercyclicity of $T$ by means of applying a criterion of
Bayart and Grivaux \cite{bg} in terms of unimodular point spectrum.
We construct $g$, $h$ and $\K$ with the required properties using a
result of Belov \cite{belov} on the distribution of values of
functions defined by lacunary trigonometric series. Note also that
the described scheme immediately produces a hypercyclic rank $1$
perturbation of a Hilbert space contraction. Indeed, if $P$ is the
orthoprojection of $L_2(\T)$ onto $\K$, then
$T=(PU)\bigr|_{\K}+(PS)\bigr|_\K$, $(PU)\bigr|_{\K}$ is a
contraction on $\K$ and $(PS)\bigr|_\K$ is a rank 1 operator on
$\K$. Thus we have the following corollary, which is of independent
interest.

\begin{corollary}\label{main00} There exist a contraction $A$ and a
bounded rank $1$ linear operator $S$ acting on a Hilbert space $\H$
such that $T=A+S$ is hypercyclic.
\end{corollary}

The following lemma summarizes the properties of $h$ and $g$ we need
in order to run the described procedure. This is the key ingredient
in the proof of Theorem~\ref{main}.

\begin{lemma}\label{main0} There exist $h,g \in L_2(\T)$ and a perfect
compact set $K\subset \T$ such that
\begin{align}
&\text{$\lambda\mapsto h_\lambda$ is a continuous map from $K$ to
$L_2(\T)$, where $\textstyle
h_\lambda(z)=\frac{h(z)}{\lambda-z};$}\label{m01}
\\
&\langle h,g_1\rangle=0,\ \ \text{$\langle h_\lambda,g\rangle=1$ and
$\langle h_\lambda,g_1\rangle=\lambda^{-1}$ for each $\lambda\in K$,
where $g_1(z)=zg(z).$}\label{m02}
\end{align}
\end{lemma}

In Section~\ref{s2}, Theorem~\ref{main} is reduced to
Lemma~\ref{main0}. The latter is proved in Section~\ref{s3}. We
discuss further possibilities in Section~\ref{s5}.

\section{Reduction of Theorem~\ref{main} to Lemma~\ref{main0} \label{s2}}

In this section we assume Lemma~\ref{main0} to be true and prove
Theorem~\ref{main}. We start by deriving the following lemma from
Lemma~\ref{main0}.

\begin{lemma}\label{main1} There exist a Hilbert space $\H$, a
unitary operator $U\in L(\H)$, $h\in\H$, $S\in L(\H)$ with
$S(\H)=\spann\{h\}$, a perfect compact set $K\subseteq \T$ and a
continuous map $\lambda\mapsto h_\lambda$ from $K$ to
$\H\setminus\{0\}$ such that
\begin{align}
&(U+S)h_\lambda=\lambda h_\lambda\quad\text{for each $\lambda\in
K;$}\label{e1}
\\
&\text{either \ $h,U^{-1}h\in\K$ \ or \ $h,U^{-1}h\notin\K$, \ where
\ $\K=\spannn\{h_\lambda:\lambda\in K\}.$}\label{e2}
\end{align}
\end{lemma}

\begin{proof}Let $\H=L_2(\T)$ and $U\in L(\H)$, $Uf(z)=zf(z)$. Obviously, $U$ is
unitary. Let also $K\subset\T$ and $h,g\in \H$ be the perfect
compact set and the functions provided by Lemma~\ref{main0}. For
$\lambda\in K$, let $h_\lambda(z)=\frac{h(z)}{\lambda-z}$. According
to (\ref{m01}), $h_\lambda\in \H$ for each $\lambda\in K$ and the
map $\lambda\mapsto h_\lambda$ from $K$ to $\H$ is continuous. By
(\ref{m02}), $h_\lambda\neq 0$ for every $\lambda\in K$. Define
$S\in L(\H)$ by the formula $Sf=\langle f,g\rangle h$. By
(\ref{m02}), $g\neq 0$ and therefore $S(\H)=\spann\{h\}$. It remains
to verify (\ref{e1}) and (\ref{e2}).

Taking into account the specific shape of $h_\lambda$ and $U$, one
can easily see that
\begin{equation}\label{uhl}
Uh_\lambda=\lambda h_\lambda - h\ \ \ \text{and}\ \ \
U^{-1}h_\lambda=\lambda^{-1}h_\lambda+\lambda^{-1}h\ \ \ \text{for
each $\lambda\in K$.}
\end{equation}
By (\ref{m02}), $\langle h_\lambda,g\rangle=1$ and therefore
$Sh_\lambda=h$ for every $\lambda\in K$. Thus the first equality in
(\ref{uhl}) implies that $(U+S)h_\lambda=\lambda h_\lambda$ for each
$\lambda\in K$. That is, (\ref{e1}) is satisfied. In order to prove
(\ref{e2}) it suffices to verify that $h\in\K$ if and only if
$U^{-1}h\in \K$.

First, assume that $h\in\K$. Then there exists a sequence
$\Bigl\{\sum\limits_{j=1}^{k_n}
c_{j,n}h_{\lambda_{j,n}}\Bigr\}_{n\in\N}$ with $\lambda_{j,n}\in K$
and $c_{j,n}\in\C$ such that
\begin{equation}\label{lico1}
\smash{\sum\limits_{j=1}^{k_n} c_{j,n}h_{\lambda_{j,n}}\to h\ \
\text{in $\H$ as $n\to\infty$.}}\PLA
\end{equation}
By (\ref{m02}), $\langle h_\lambda,Ug\rangle=\lambda^{-1}$ for any
$\lambda\in K$ and $\langle h,Ug\rangle=0$. Using these equalities
and taking the scalar product with $Ug$ in (\ref{lico1}), we obtain
\begin{equation}\label{lico11}
\smash{\sum\limits_{j=1}^{k_n} \frac{c_{j,n}}{\lambda_{j,n}}\to 0\ \
\text{as $n\to\infty$.}}\PLA
\end{equation}
Applying $U^{-1}$ to (\ref{lico1}), we get
$\smash{\sum\limits_{j=1}^{k_n} c_{j,n}U^{-1}h_{\lambda_{j,n}}\to
U^{-1}h}$ in $\H$ as $n\to\infty$. Using the second equality in
(\ref{uhl}), we obtain
\begin{equation*}
\smash{\biggl(\sum\limits_{j=1}^{k_n}\frac{c_{j,n}}{\lambda_{j,n}}\biggr)U^{-1}h+
\sum\limits_{j=1}^{k_n}
\frac{c_{j,n}}{\lambda_{j,n}}h_{\lambda_{j,n}}\to U^{-1}h\ \
\text{in $\H$ as $n\to\infty$.}}\PLA
\end{equation*}
By (\ref{lico11}), $\smash{\sum\limits_{j=1}^{k_n}
\frac{c_{j,n}}{\lambda_{j,n}}h_{\lambda_{j,n}}\to U^{-1}h}$ as
$n\to\infty$. Hence $U^{-1}h\in\K$. Thus $h\in\K$ implies
$U^{-1}h\in\K$.

Now we assume that $U^{-1}h\in\K$. Then there exists a sequence
$\Bigl\{\sum\limits_{j=1}^{k_n}
c_{j,n}h_{\lambda_{j,n}}\Bigr\}_{n\in\N}$ with $\lambda_{j,n}\in K$
and $c_{j,n}\in\C$ such that
\begin{equation}\label{lico2}
\smash{\sum\limits_{j=1}^{k_n} c_{j,n}h_{\lambda_{j,n}}\to U^{-1}h\
\ \text{in $\H$ as $n\to\infty$.}}\PLA
\end{equation}
By (\ref{m02}), $\langle h_\lambda,g\rangle=1$ for any $\lambda\in
K$ and $\langle U^{-1}h,g\rangle=\langle h,Ug\rangle=0$. Using these
equalities and taking the scalar product with $g$ in (\ref{lico1}),
we obtain
\begin{equation}\label{lico22}
\smash{\sum\limits_{j=1}^{k_n} c_{j,n}\to 0\ \ \text{as
$n\to\infty$.}}\PLA
\end{equation}
Applying $U$ to (\ref{lico2}), we get $\sum\limits_{j=1}^{k_n}
c_{j,n}Uh_{\lambda_{j,n}}\to h$ in $\H$ as $n\to\infty$. Using the
first equality in (\ref{uhl}), we see that
\begin{equation*}
\smash{-\biggl(\sum\limits_{j=1}^{k_n}c_{j,n}\biggr)h+
\sum\limits_{j=1}^{k_n} \lambda_{j,n}c_{j,n}h_{\lambda_{j,n}}\to h\
\ \text{in $\H$ as $n\to\infty$.}}\PLA
\end{equation*}
By (\ref{lico22}), $\sum\limits_{j=1}^{k_n}
\lambda_{j,n}c_{j,n}h_{\lambda_{j,n}}\to h$ as $n\to\infty$. Hence
$h\in\K$. Thus $U^{-1}h\in\K$ implies that $h\in\K$.
\end{proof}

We also need the following criterion of hypercyclicity by Bayart and
Grivaux \cite{bg}.

\begin{thmbg}\it Let $X$ be a separable infinite dimensional Banach
space, $T\in L(X)$ and assume that there exists a continuous Borel
probability measure $\nu$ on the unit circle $\T$ such that for each
Borel set $A\subseteq \T$ with $\nu(A)=1$, the space
\begin{equation}\label{bgsp}
\smash{X_A=\spann\left(\bigcup_{z\in A}\ker(T-zI)\right)}\PLA
\end{equation}
is dense in $X$. Then $T$ is hypercyclic.
\end{thmbg}\rm

\begin{corollary}\label{bg} Let $X$ be a separable infinite dimensional Banach
space, $T\in L(X)$ and assume that there exists a perfect compact
set $K\subseteq\T$ and a continuous map $\lambda\mapsto x_\lambda$
from $K$ to $X$ such that $Tx_\lambda=\lambda x_\lambda$ for each
$\lambda\in K$ and $\spann\{x_\lambda:\lambda\in K\}$ is dense in
$X$. Then $T$ is hypercyclic.
\end{corollary}

\begin{proof} Since $K$ is a perfect compact subset of $\T$, we can
pick a continuous Borel probability measure $\nu$ on the unit circle
$\T$ such that $K$ is exactly the support of $\nu$. Let now
$A\subseteq \T$ be a Borel measurable set such that $\nu(A)=1$.
Since $K$ is the support of $\nu$, $B=A\cap K$ is dense in $K$.
Clearly $x_\lambda\in X_A$ for each $\lambda\in B$, where $X_A$ is
defined in (\ref{bgsp}). Thus $\spann\{x_\lambda:\lambda\in
B\}\subseteq X_A$. Since the map $\lambda\mapsto x_\lambda$ is
continuous, $B$ is dense in $K$ and $\spann\{x_\lambda:\lambda\in
K\}$ is dense in $X$, we see that $\spann\{x_\lambda:\lambda\in B\}$
is dense in $X$. Hence $X_A$ is dense in $X$. By Theorem~BG, $T$ is
hypercyclic.
\end{proof}

\begin{lemma}\label{te} Let $U$ be a unitary operator acting on a
Hilbert space $\H$ and $\K$, $\K_+$ and $\K_-$ be closed linear
subspaces of $\H$ such that $\K\subseteq \K_+\cap \K_-$,
$\dim\K_+/\K=\dim\K_-/\K=1$, $U(\K)\subseteq\K_+$,
$U^{-1}(\K)\subseteq\K_-$, $U(\K)\not\subseteq \K$ and
$U^{-1}(\K)\not\subseteq \K$. Then there exist a unitary operator
$V\in L(\K)$ and a bounded linear operator $A:\K\to\H$ of rank at
most $1$ such that $U\bigr|_{\K}=V+A$.
\end{lemma}

\begin{proof} Let $X=U^{-1}(\K)\cap \K$ and $Y=U(\K)\cap \K$.
Clearly $X$ and $Y$ are closed linear subspaces of $\K$. Moreover,
$U(X)=\K\cap U(\K)=Y$.

Since $U(\K)\subseteq\K_+$ and $\dim\K_+/\K=1$, we have $\dim
\K/X\leq 1$. Similarly, since $U^{-1}(\K)\subseteq\K_-$ and
$\dim\K_-/\K=1$, we see that $\dim \K/Y\leq 1$. On the other hand,
the relations $U(\K)\not\subseteq \K$ and $U^{-1}(\K)\not\subseteq
\K$ imply that $X\neq\K$ and $Y\neq\K$. Thus $\dim \K/X=\dim\K/Y=1$.
Now we can pick $x,y\in\K$ such that $\|x\|=\|y\|=1$, $x$ is
orthogonal to $X$, $y$ is orthogonal to $Y$ and
$\K=X\oplus\spann\{x\}=Y\oplus\spann\{y\}$. Define the operator
$V:\K\to \H$ be the formula
$$
Vu=Uu+\langle u,x\rangle(y-Ux),\quad u\in\K.
$$
It is easy to see that $U\bigr|_X=V\bigr|_X$ and $Vx=y$. Since
$U(X)=Y$, $V$ maps $X$ isometrically onto $Y$. Since $Vx=y$, $x$
spans the orthocomplement of $X$ and $y$ spans the orthocomplement
of $Y$, we see that $V$ maps $\K$ onto itself isometrically. Thus
$V\in L(\K)$ is a unitary operator. It remains to notice that
according to the last display, $U\bigr|_{\K}=V+A$, where the bounded
linear operator $A:\K\to\H$ is given by the formula $Au=\langle
u,x\rangle(y-Ux)$ and therefore has rank at most $1$.
\end{proof}

\begin{lemma} \label{ge} Let $U$ be a unitary operator acting on a
Hilbert space $\H$, $h\in\H$, $S\in L(\H)$ with
$S(\H)\subseteq\spann\{h\}$ and $\K$ be a closed linear subspace of
$\H$ invariant for the operator $U+S$. Assume also that $(U+S)(\K)$
is dense in $\K$ and either $h,U^{-1}h\in\K$ or $h,U^{-1}h\notin
\K$. Then the restriction $T\in L(\K)$ of $U+S$ to $\K$ can be
expressed as $T=V+R$, where $V\in L(\K)$ is a unitary operator on
$\K$ and $R\in L(\K)$ has rank at most $2$.
\end{lemma}

\begin{proof} If $x\in\K$,
then $Ux=Tx-Sx\in \K_+=\spann(\K\cup\{h\})$. Thus
$U(\K)\subseteq\K_+$. Applying $U^{-1}$ to the equality $Ux=Tx-Sx$,
we obtain $U^{-1}Tx=x+U^{-1}Sx\in \K_-=\spann(\K\cup\{U^{-1}h\})$
for each $x\in \K$. Since $T$ has dense range, $U^{-1}(\K)\subseteq
\K_-$.

If $h\in\K$ and $U^{-1}h\in\K$ , then $\K_+=\K_-=\K$ and therefore
$\K$ is an invariant subspace for $U$ and $U^{-1}$. Hence the
restriction $V\in L(\K)$ of $U$ to $\K$ is unitary and $T=V+R$ with
$R=S\bigr|_{\K}$ being of rank at most 1. If $S\bigr|_{\K}=0$, then
$T$ is the restriction of $U$ to $\K$ and therefore $T$ is an
isometry. Since $T$ also has dense range, $T$ is unitary. Thus $T$
has the required shape with $V=T$ and $R=0$.

It remains to consider the case $h\notin\K$, $U^{-1}h\notin\K$ and
$S\bigr|_{\K}\neq 0$. Since $h\notin\K$ and $U^{-1}h\notin\K$, $\K$
is a  closed hyperplane in $\K_+$ and in $\K_-$. Since
$S\bigr|_{\K}\neq 0$, $S(\H)\subseteq\spann\{h\}$ and
$h,U^{-1}h\notin\K$, the equalities $Ux=Tx-Sx$ and
$U^{-1}Tx=x+U^{-1}Sx$ for $x\in\K$ imply that $U(\K)\not\subseteq
\K$ and $U^{-1}(\K)\not\subseteq \K$. Thus all conditions of
Lemma~\ref{te} are satisfied. By Lemma~\ref{te}, there is a unitary
operator $V\in L(\K)$ and a bounded linear operator $A:\K\to\H$ of
rank at most $1$ such that $U\bigr|_{\K}=V+A$. Thus $T=V+R$, where
$R=A+S\bigr|_{\K}$. Clearly $R=T-V$ takes values in $\K$ and has
rank at most $2$ as a sum of two operators $A$ and $S\bigr|_{\K}$
from $\K$ to $\H$ of rank at most $1$.
\end{proof}

\subsection{Proof of Theorem~\ref{main} modulo Lemma~\ref{main0}}

Lemma~\ref{main1} guarantees the existence of a unitary operator $U$
acting on a Hilbert space $\H$, $h\in\H$, $S\in L(\H)$ with
$S(\H)=\spann\{h\}$, a perfect compact subset $K$ of $\T$ and a
continuous map $\lambda\mapsto h_\lambda$ from $K$ to
$\H\setminus\{0\}$ such that (\ref{e1}) and (\ref{e2}) are
satisfied.

Let $\K$ be the space defined in (\ref{e2}). According to
(\ref{e1}), $\K$ is invariant for $U+S$. Let $T\in L(\K)$ be the
restriction of $U+S$ to $\K$. By (\ref{e1}), $Th_\lambda=\lambda
h_\lambda$ and therefore $h_\lambda$ are linearly independent for
$\lambda\in K$. By definition of $\K$, $\spann\{h_\lambda:\lambda\in
\K\}$ is a dense subspace of $\K$. Thus $\K$ is separable and
infinite dimensional. Corollary~\ref{bg} implies that $T$ is
hypercyclic.

On the other hand, the equalities $Th_\lambda=\lambda h_\lambda$
imply that $\spann\{h_\lambda:\lambda\in \K\}$ is contained in
$T(\K)$ and therefore $T(\K)$ is dense in $\K$. Then (\ref{e2}) and
Lemma~\ref{ge} imply that $T$ is a sum of a unitary operator and an
operator of rank at most 2 as required in Theorem~\ref{main}.

\section{Lemma~\ref{main0}: preparation and proof \label{s3}}

To make the idea of the proof of Lemma~\ref{main0} more transparent,
we note that the scalar product of the functions $f_1,f_2\in
L_2(\T)$ can be written in terms of a contour integral: $\langle
f_1,f_2\rangle=\frac1{2\pi
i}\ssub{\oint}{\T}\frac{f_1(z)\overline{f_2(z)}}{z}\,dz$. Thus
condition (\ref{m02}) reads as
$$
\oint_\T \frac{h(w)\overline{g(w)}}{w^2}\,dw=0,\quad \oint_\T
\frac{h(w)\overline{g(w)}}{(z-w)w}\,dw=2\pi i\ \ \text{and}\ \
\oint_\T \frac{h(w)\overline{g(w)}}{(z-w)w^2}\,dw=\frac{2\pi i}z\ \
\text{for $z\in K$}.
$$
Assuming that the function $\psi(z)=\frac{h(z)\overline{g(z)}}{2\pi
iz}$ is continuous and vanishes on $K$, the above display can be
rewritten as
$$
\oint_\T \frac{\psi(w)}{w}\,dw=0,\quad \oint_\T
\frac{\frac{\psi(w)}{w}-\frac{\psi(z)}{z}}{(z-w)}\,dw=1\ \
\text{and}\ \ \oint_\T \frac{\psi(w)-\psi(z)}{(z-w)w}\,dw=z^{-1}\ \
\text{for $z\in K$}.
$$
We prove Lemma~\ref{main0} by constructing $K$ and an appropriate
function $\psi$ and then splitting it into a product to recover $h$
and $g$.

\subsection{Auxiliary results}

The next few lemmas certainly represent known facts. We state them
in a convenient for our purposes form, different from the one
usually found in the literature. For the sake of completeness we
sketch their proofs. For a subset $A$ of a metric space $(M,d)$, the
symbol $\dist(x,A)$ stands for the distance from $x\in M$ to $A$:
$\smash{\dist(x,A)=\inf\limits_{y\in A}d(x,y)}$. Speaking of $\T$,
we always assume that it carries the metric inherited from $\C$.

\begin{lemma}\label{comp11} Let $F$ be an uncountable closed
subset of $\T$. Then there exists a perfect compact set $K\subset F$
such that $f_\alpha\in L_2(\T)$ for any $\alpha\in(0,1/2)$, where
$f_\alpha(z)=(\dist(z,K))^{-\alpha}$.
\end{lemma}

The above lemma immediately follows from the next result.

\begin{lemma}\label{comp} Let $[a,b]$ be a bounded closed interval
in $\R$ and $F$ be an uncountable closed subset of $[a,b]$. Then
there exists a perfect compact set $K\subset F$ such that
\begin{equation}\label{int1}
\smash{\int_a^b (\dist(x,K))^{-\alpha}\lambda(dx)<\infty,\ \ \
\text{for each $\alpha<1$},}\PLA
\end{equation}
where $\lambda$ is the Lebesgue measure on the real line.
\end{lemma}

\begin{proof} For a subset $A$ of the real line, we say that $x\in\R$ is a {\it
left accumulation point} for $A$ if $(x-\epsilon,x)\cap A$ is
uncountable for any $\epsilon>0$. Similarly $x$ is a {\it right
accumulation point} for $A$ if $(x,x+\epsilon)\cap A$ is uncountable
for any $\epsilon>0$. It is a well-known fact and an easy exercise
that for any uncountable subset $A$ of $\R$, all points of $A$
except for countably many are left and right accumulation points of
$A$. We construct $K$ by means of a procedure similar to the one
used to construct the standard Cantor set. For each $n\in\N$, let
$\Omega_n=\{0,1\}^n$ be endowed with the lexicographical ordering:
$\epsilon<\epsilon'$ if and only if $\sum\limits_{j=1}^n
\epsilon_j2^{n-j}<\sum\limits_{j=1}^n \epsilon'_j2^{n-j}$. Using the
fact that all points of $F$, except for countably many, are left and
right accumulation points for $F$, we can easily construct
(inductively with respect to $n$) elements
$a_\epsilon^n,b_\epsilon^n\in F$ for $n\in\N$ and
$\epsilon\in\Omega_n$ such that:
\begin{align}
&a^{n+1}_{\epsilon,0}=a^{n}_{\epsilon}\ \ \text{and}\ \
b^{n+1}_{\epsilon,1}=b^{n}_{\epsilon}\ \ \text{for any $n\in\N$ and
$\epsilon\in\Omega_n$};\label{com1}
\\
&\text{$a_\epsilon^n<b_\epsilon^n<a^n_{\epsilon'}<b^n_{\epsilon'}$
for $n\in\N$, $\epsilon,\epsilon'\in\Omega_n$,
$\epsilon<\epsilon'$;}\label{com2}
\\
&\text{$b^n_{\epsilon}-a^n_{\epsilon}<\frac1{n!}$ for any $n\in\N$
and $\epsilon\in\Omega_n$;}\label{com3}
\\
&\text{$a_\epsilon^n$ is a right accumulation point for $F$ for any
$n\in\N$ and $\epsilon\in\Omega_n$;}\label{com4}
\\
&\text{$b_\epsilon^n$ is a left accumulation point for $F$ for any
$n\in\N$ and $\epsilon\in\Omega_n$.}\label{com5}
\end{align}
We do not really need conditions (\ref{com4}) and (\ref{com5}) in
what follows. They are included in order to enable us to run the
inductive procedure. Now we can define
\begin{equation}\label{K}
\smash{K=\bigcap\limits_{n=1}^\infty
\bigcup_{\epsilon\in\Omega_n}[a_\epsilon^n,b_\epsilon^n].}\PLA
\end{equation}
Compactness and non-emptiness of $K$ are obvious. Actually, $K$ is
homeomorphic to $\{0,1\}^\N$ with the 2-element space $\{0,1\}$
carrying the discrete topology (=homeomorphic to the standard Cantor
set). Indeed, the map from $\{0,1\}^\N$ to $K$, which sends a
$0{-}1$ sequence $\{\epsilon_1,\epsilon_2,\dots\}$ to the unique
common point of the nested sequence
$[a^n_{\epsilon_1,\dots,\epsilon_n},b^n_{\epsilon_1,\dots,\epsilon_n}]$
of closed intervals is a homeomorphism. Thus $K$ is perfect. The
above observations show also that the set $A=\{a^n_\epsilon:n\in\N,\
\epsilon\in\Omega_n\}$ is dense in $K$. Since $A\subset F$ and $F$
is closed, $K\subset F$. It remains to show that (\ref{int1}) is
satisfied. According to (\ref{com3}),
$$
\lambda\left(
\bigcup_{\epsilon\in\Omega_n}[a_\epsilon^n,b_\epsilon^n] \right)
<\frac{2^n}{n!}\to 0\ \ \text{as}\ \ n\to \infty.
$$
By (\ref{K}), $\lambda(K)=0$. Clearly $[a,b]\setminus K$ is the
union of disjoint open intervals $I_0=(a,\alpha^1_0)$,
$I_1=(\beta^1_1,b)$ and
$J^n_j=(\beta^n_{\epsilon^j},\alpha^n_{\epsilon^{j+1}})$ for
$n\in\N$ and $1\leq j\leq 2^{n}-1$, where
$\Omega_n=\{\epsilon^1,\dots,\epsilon^{2^n}\}$,
$\epsilon^1<{\dots}<\epsilon^{2^n}$. Condition (\ref{com3}) and the
fact that each $J_j^n$ is contained in one of the intervals of the
shape $[a_\epsilon^{n-1},b_\epsilon^{n-1}]$ implies that the length
$\lambda(J^n_j)$ satisfies $\lambda(J^n_j)<\frac1{(n-1)!}$ for
$n\geq 2$. Fix $\alpha<1$. Direct calculation shows that the
function $\dist(\cdot,K)^{-\alpha}$ is integrable on $I_0$, $I_1$
and each of $J^n_j$ and
\begin{equation}\label{jnj}
 \int_{J^n_j}
(\dist(x,K))^{-\alpha}\lambda(dx)=2\int_0^{\lambda(J^n_j)/2}
t^{-\alpha}\,dt=\frac{2^\alpha
(\lambda(J^n_j))^{1-\alpha}}{1-\alpha}.
\end{equation}
Since $\lambda(K)=0$, we see that (\ref{int1}) is equivalent to
$$
\sum_{n=2}^\infty\sum_{j=1}^{2^n-1} \int_{J^n_j}
(\dist(x,K))^{-\alpha}\lambda(dx)<\infty.
$$
Since $\lambda(J^n_j)<\frac1{(n-1)!}$, from (\ref{jnj}) it follows
that convergence of the above series reduces to convergence of
$\sum\limits_{n=2}^\infty \frac{2^n}{((n-1)!)^{1-\alpha}}$. Thus
(\ref{int1}) is satisfied.
\end{proof}

For $0<\alpha\leq 1$, the symbol $H_\alpha(\T)$ stands for the space
of functions $f:\T\to\C$ satisfying the H\"older condition with the
exponent $\alpha$. That is, $f\in H_\alpha(\T)$ if and only if there
is $C>0$ such that $|f(z)-f(w)|\leq C|z-w|^\alpha$ for all
$z,w\in\T$. The next lemma provides a formula, which is a variant of
the integral formula for the conjugate function. It is, of course,
true under much weaker restrictions on the function involved.

\begin{lemma}\label{conj1}Let $\{a_n\}_{n\in\Z}$ be a sequence of
complex numbers such that
$\sum\limits_{n=-\infty}^\infty|a_n|<\infty$ and the function
$f:\T\to\C$, $f(z)=\sum\limits_{n=-\infty}^\infty a_nz^n$ belongs to
$H_\alpha(\T)$ for some $\alpha\in (0,1]$. Then for each $z\in\T$,
the function $w\mapsto \frac{w(f(w)-f(z))}{z-w}$ is Lebesgue
integrable and
\begin{equation}\label{conj2}
\smash{\int_\T \frac{w(f(w)-f(z))}{z-w}\,\mu(dw)=f_-(z),\ \ \
\text{where}\ \ f_-(z)=\sum\limits_{n=-\infty}^{-1} a_nz^n.}\PLAN
\end{equation}
\end{lemma}

\begin{proof}Since $f\in H_\alpha(\T)$, there exists $C>0$ such that
$|f(z)-f(w)|\leq C|z-w|^\alpha$ for any $z,w\in\T$. For fixed
$z\in\T$, let $f_z(w)=\frac{w(f(w)-f(z))}{w-z}$. Then $|f_z(w)|\leq
C|z-w|^{\alpha-1}$. It immediately follows that $f_z\in L_1(\T)$. It
remains to verify (\ref{conj2}). First, observe that (\ref{conj2})
is equivalent to
\begin{equation}\label{conj3}
\smash{\frac1{2\pi i}\oint_\T
\frac{f(w)-f(z)}{z-w}\,dw=f_-(z),}\PLAN
\end{equation}
where the contour $\T$ is encircled counterclockwise. Using the
Cauchy formula, one can easily show that for any $n\in\Z$,
$$
\frac1{2\pi i}\oint_\T \frac{w^n-z^n}{z-w}\,dw=
\left\{\begin{array}{ll}0&\text{if $n\geq 0$},\\
z^{n}&\text{if $n<0$.}\end{array}\right.
$$
It follows that (\ref{conj3}) and therefore (\ref{conj2}) is
satisfied for $f$ being a trigonometric polynomial (=a Laurent
polynomial). Now consider the sequence $\{p_n\}_{n\in\N}$ of Fej\'er
sums for $f$:
$$
p_n(z)=\sum_{j=-n}^n \Bigl(1-\frac{|j|}{n+1}\Bigr)a_jz^j.
$$
Clearly $\{p_n\}$ converges to $f$ uniformly on $\T$ as
$n\to\infty$. On the other hand, since $p_n$ is the convolution of
$f$ with the $n^{\rm th}$ Fej\'er kernel \cite{z} and the latter is
positive and has integral $1$, we immediately have
$|p_n(z)-p_n(w)|\leq C|z-w|^\alpha$ for any $z,w\in\T$ and any
$n\in\N$ (the continuity modulus of any Fej\'er sum of a continuous
function on $\T$ does not exceed the continuity modulus of the
function itself). Hence for each $w\in\T$, $w\neq z$, we have
$$
\frac{w(p_n(w)-p_n(z))}{z-w}\to \frac{w(f(w)-f(z))}{z-w}\ \
\text{and}\ \ \biggl|\frac{w(p_n(w)-p_n(z))}{z-w}\biggr|\leq
C|z-w|^{\alpha-1}\ \ \text{for any $n\in\N$.}
$$
Applying the Lebesgue dominated convergence theorem and the fact
that (\ref{conj2}) is true for trigonometric polynomials, we obtain
\[
\int_\T \frac{w(f(w)-f(z))}{z-w}\,\mu(dw)=\lim_{n\to\infty}\int_\T
\frac{w(p_n(w)-p_n(z))}{z-w}\,\mu(dw)=\lim_{n\to\infty}(p_n)_-(z)=f_-(z).\qedhere
\]
\end{proof}

\begin{lemma}\label{funct} Let $a\in\R$ and $b\in\N$ be such that $b>a>1$ and
$f:\T\to \C$ be defined by the formula
$$
\smash{f(z)=\sum_{n=1}^\infty a^{-n}z^{b^n}.}\PLA
$$
Then $f\in H_{\alpha}(\T)$, where $\alpha=\log_b a$.
\end{lemma}

\begin{proof} Let $z,w\in\T$,
$z\neq w$. Pick $m\in\N$ such that $b^{-m}\leq\frac{|z-w|}2\leq
b^{1-m}$. Clearly $|f(z)-f(w)|\leq \sum\limits _{j=1}^\infty
a^{-j}|z^{b^j}-w^{b^j}|$. Using the estimate $|z^{b^j}-w^{b^j}|\leq
2$  for $j\geq m$ and the estimate $|z^{b^j}-w^{b^j}|\leq
b^{j}|z-w|$ for $j<m$, we obtain
$$
\smash{|f(z)-f(w)|\leq|z-w|\sum_{j=1}^{m-1}a^{-j}b^{j}+2\sum_{j=m}^\infty
a^{-j}\leq \frac{a|z-w|}{b-a}(b/a)^m+\frac{2a}{a-1}a^{-m}.}\PLA
$$
Since $b^{-m}\leq\frac{|z-w|}2\leq b^{1-m}$, we have
$a^{-m}=(b^{-m})^\alpha\leq 2^{-\alpha}|z-w|^\alpha$ and
$(b/a)^m=(b^{-m})^{\alpha-1}\leq (2b)^{1-\alpha}|z-w|^{\alpha-1}$.
Hence, according to the above display, $|f(z)-f(w)|\leq
C|z-w|^{\alpha}$ with
$C=\frac{a2^{-\alpha}}{b-a}+\frac{2a(2b)^{1-\alpha}}{a-1}$. Thus
$f\in H_{\alpha}(\T)$, as required.
\end{proof}

\begin{lemma}\label{funct1} Let $\alpha\in (1/4,1]$, $f\in
H_\alpha(\T)$, $K\subset \T$  be a non-empty compact set and
$f(z)=0$ for every $z\in K$. Then for each $z\in K$, the function
$f_z(w)=\frac{f(w)\,({\tt dist}(w,K))^\alpha}{z-w}$ belongs to
$L_2(\T)$. Moreover the map $z\mapsto f_z$ from $K$ to $L_2(\T)$ is
continuous.
\end{lemma}

\begin{proof} Since $f\in H_\alpha(\T)$, there is $C>0$ such that
$|f(z)-f(w)|\leq C|z-w|^\alpha$ for any $z,w\in \T$.

Let $z\in K$. Since $f(z)=0$, we have $|f(w)|\leq C|z-w|^\alpha$ for
each $w\in\T$. Moreover, $\dist(w,K)\leq |w-z|$ for each $w\in\T$.
Thus $|f_z(w)|\leq C|z-w|^{2\alpha-1}$ for any $w\in\T$. Since
$\alpha>\frac14$ it follows that $f_z\in L_2(\T)$ for any $z\in K$.

It remains to show that the map $z\mapsto f_z$ from $K$ to $L_2(\T)$
is continuous. Clearly it is enough to show that there is
$c(\alpha)>0$ such that $\|f_z-f_s\|^2\leq
c(\alpha)|z-s|^{4\alpha-1}$ for any $z,s\in K$. In order to get rid
of the dead weight of constants, we temporarily assume the following
notation. We write $A\ll B$ if there is a constant $c$ depending on
$\alpha$ only such that $A\leq cB$. Thus we are going to show that
$\|f_z-f_s\|^2\ll |z-s|^{4\alpha-1}$. Let $z,s\in K$, $z\neq s$.
Since
\begin{align*}
&|f_z(w)-f_s(w)|=\frac{|z-s||f(w)|\,({\tt
dist}(w,K))^\alpha}{|z-w||s-w|},
\\
&\dist(w,K)\leq \min\{|w-z|,|w-s|\}\ \ \text{and\ \ $|f(w)|\leq
C(\min\{|w-z|,|w-s|\})^\alpha\!$},
\end{align*}
we see that
$$
|f_z(w)-f_s(w)|\ll
|z-s|\frac{(\min\{|w-z|,|w-s|\})^{2\alpha}\!\!}{|z-w||s-w|}.
$$
Hence
$$
\|f_z-f_s\|^2\ll |z-s|^2\int_\T
\frac{(\min\{|w-z|,|w-s|\})^{4\alpha}}{|z-w|^2|s-w|^2}\,\mu(dw).
$$
As for any two distinct points in the unit circle, for $z$ and $s$
we can find $a,b\in\R$ such that $0<b\leq \frac{\pi}{2}$ and
$\{z,s\}=\{e^{i(a+b)},e^{i(a-b)}\}$. Clearly $b\ll |z-s|\ll b$.
Using this notation, the last display and straightforward symmetry
considerations, we get
$$
\|f_z-f_s\|^2\ll b^2\int_0^\pi
\frac{|e^{it}-e^{ib}|^{4\alpha}}{|e^{it}-e^{ib}|^2|e^{it}-e^{-ib}|^2}\,dt.
$$
Since $|t-b|\ll |e^{it}-e^{ib}|\ll |t-b|$ and $|t+b|\ll
|e^{it}-e^{-ib}|\ll |t+b|$ for $t\in[0,\pi]$, we have
$$
\|f_z-f_s\|^2\ll b^2\int_0^\pi
\frac{|t-b|^{4\alpha-2}}{|t+b|^2}\,dt.
$$
We split the integration interval $[0,\pi]$ into the union of
$[0,2b]$ and $[2b,\pi]$. Since $|t+b|^{-2}\ll b^{-2}$ for $0\leq
t\leq 2b$ and $\frac{|t-b|^{4\alpha-2}}{|t+b|^2}\ll t^{4\alpha-4}$
for $2b\leq t\leq \pi$, we get
$$
b^2\!\!\int_0^{2b} \frac{|t-b|^{4\alpha-2}}{|t+b|^2}\,dt\ll\!\!
\int_0^{2b}|t-b|^{4\alpha-2}\,dt\ll b^{4\alpha-1}\ \ \text{and}\ \
b^2\!\!\int_{2b}^\pi \frac{|t-b|^{4\alpha-2}\!\!}{|t+b|^2}\,dt\ll
b^2\!\!\int_{2b}^\pi t^{4\alpha-4}\,dt\ll b^{4\alpha-1}.
$$
By the last two displays, $\|f_z-f_s\|^2\ll b^{4\alpha-1}\ll
|z-s|^{4\alpha-1}$, which completes the proof.
\end{proof}

The following Theorem is due to Belov \cite[Corollary~3.1]{belov}.

\begin{thmb} Let $\alpha,\beta>0$, $\lambda>2$, $M\geq 0$,
$\{\lambda_n\}_{n\in\N}$ be a sequence of positive numbers,
$\{a_n\}_{n\in\N}$ be a sequence of complex numbers and $g:\R\to\C$
be such that
\begin{align*}
&\text{$|g(x)-g(y)|\leq M|x-y|$ \ for any \ $x,y\in \R$, \
$\textstyle \sum\limits_{n=1}^\infty |a_n|<\infty$, \
$\alpha(1+\beta)\leq 1$,}
\\
&\text{$\textstyle \frac{\lambda_{m+1}}{\lambda_m}\geq \lambda$, \
$\textstyle |a_{m}|\leq
\beta\!\!\!\!\sum\limits_{n=m+1}^\infty|a_n|$ \ and \ $\textstyle
2\pi\frac{\lambda-1}{\lambda-2}\Bigl(M+\sum\limits_{n=1}^m|a_n|\lambda_n\Bigr)
\leq \alpha|a_{m+1}|\lambda_{m+1}$ \ for each $m\in\N$}.
\end{align*}
Assume also that $x_0\in\R$, $\phi:\R\to\C$ is defined by the
formula
$$
\text{$\textstyle \phi(x)=g(x)+\sum\limits_{n=1}^\infty
a_ne^{i\lambda_n x}$ and $I=[x_0-\Delta,x_0+\Delta]$, \ where
$\textstyle \Delta=\frac{2\pi\lambda}{(\lambda-2)\lambda_1}$}.
$$
Then $\phi^{-1}(w)\cap I$ is uncountable for any $w\in\C$ satisfying
$$
\text{$\textstyle \frac\beta{1+\beta}\sum\limits_{n=1}^\infty
|a_n|\leq |g(x_0)-w|\leq (1-\alpha)\sum\limits_{n=1}^\infty |a_n|$.}
$$
\end{thmb}

\begin{remark}\label{rem} The main point in \cite{belov} is to find
$\phi:\R\to\C$ defined by an absolutely convergent lacunary
trigonometric series with the continuity modulus as small as
possible and with $\phi(\R)$ having non-empty interior in $\C$. The
latter means that $\phi$ defines a Peano curve. Belov's construction
allows not only to ensure that certain complex numbers belong to
$\phi(\R)$ but also that they are attained by $\phi$ uncountably
many times. We take an advantage of the latter property.
\end{remark}

\subsection{Proof of Lemma~\ref{main0}}

Consider the functions
\begin{equation}\label{gamma}
\gamma,\psi:\T\to \C,\quad\smash{\gamma(z)=\sum_{n=1}^\infty
8^{1-n}z^{2^{9n}}\ \ \text{and}\ \
\psi(z)=\gamma(z)+\gamma(z^{-1})=\sum_{n=1}^\infty
8^{1-n}(z^{2^{9n}}+z^{-2^{9n}}).} \PLA
\end{equation}
Since $\log_{2^9}8=1/3$, Lemma~\ref{funct} implies that $\gamma\in
H_{1/3}(\T)$. Hence $\psi\in H_{1/3}(\T)$. If $\phi:\R\to\C$ is
defined by the formula
$$
\phi(x)=\gamma(e^{2\pi ix})=\sum_{n=1}^\infty 8^{1-n}e^{2\pi 2^{9n}
ix},
$$
then $\phi$ is $2\pi$-periodic and has the shape exactly as in
Theorem~B with $g=0$, $a_n=8^{1-n}$ and $\lambda_n=2\pi 2^{9n}$. Now
we put $M=0$, $\lambda=2^9$, $\beta=7$ and $\alpha=1/8$. It is
straightforward to verify that all conditions of Theorem~B are
satisfied. Since $\frac\beta{\beta+1}\sum\limits_{n=1}^\infty
|a_n|=(1-\alpha)\sum\limits_{n=1}^\infty |a_n|=1$, Theorem~B implies
that $\phi^{-1}(w)$ is uncountable if $|w|=1$. Hence
$\gamma^{-1}(w)$ is uncountable for each $w\in\T$. In particular,
\begin{equation}\label{unc}
\text{the closed set\quad $F=\{z\in\T:\gamma(z)=i\}$\quad is
uncountable.}
\end{equation}
By (\ref{unc}) and Lemma~\ref{comp11}, there is a perfect compact
set $K\subset\T$ such that $K\subseteq F$ and
$\dist(\cdot,K)^{-\alpha}\in L_2(\T)$ for each $\alpha<\frac12$. In
particular, $g\in L_2(\T)$, where
\begin{equation}\label{g}
g(z)=-iz^{-1}\dist(z,K)^{-1/3}.
\end{equation}
Obviously, $h\in L_2(\T)$, where
\begin{equation}\label{h}
h(z)=\psi(z)\,\dist(z,K)^{1/3}
\end{equation}
and $\psi$ is defined in (\ref{gamma}). In order to prove
Lemma~\ref{funct}, it suffices to verify that conditions (\ref{m01})
and (\ref{m02}) are satisfied for the just specified $K$, $h$ and
$g$.

First, observe that $\gamma(z)=i$ and therefore
$\gamma(z^{-1})=\gamma(\overline{z})=\overline{\gamma(z)}=-i$ for
$z\in F$. Thus using (\ref{gamma}) and the inclusion $K\subseteq F$,
we get
\begin{equation}\label{ii}
\gamma(z)=i,\ \ \gamma(z^{-1})=-i\ \ \text{and}\ \ \psi(z)=0\ \ \
\text{for each $z\in K$.}
\end{equation}
Since $\psi\in H_{1/3}(\T)$, (\ref{h}), (\ref{ii}) and
Lemma~\ref{funct1} imply that $h_z\in L_2(\T)$ for each $z\in K$,
where $h_z(w)=\frac{h(w)}{z-w}$ and the map $z\mapsto h_z$ from $K$
to $L_2(\T)$ is continuous. Thus (\ref{m01}) is satisfied. It
remains to verify (\ref{m02}). First, from (\ref{h}) and (\ref{g})
it follows that for each $z\in K$,
$$
\langle h_z,g\rangle=\!\!\int_\T h_z(w)\overline{g(w)}\,\mu(dw)=
i\!\!\int_\T\frac{\psi(w)\,\dist(w,K)^{1/3}\!\!}{z-w}w\,\dist(w,K)^{-1/3}\,\mu(dw)=
i\!\!\int_\T\frac{w\psi(w)}{z-w}\,\mu(dw).
$$
Applying Lemma~\ref{conj1}, we see that $\langle
h_z,g\rangle=i\psi_-(z)$. According to (\ref{gamma}),
$\psi_-(z)=\gamma(z^{-1})$. By (\ref{ii}), $\gamma(z^{-1})=-i$ and
therefore $\langle h_z,g\rangle=i\psi_-(z)=i(-i)=1$. Thus
\begin{equation}\label{p1}
\langle h_z,g\rangle=1\ \ \ \text{for each $z\in K$.}
\end{equation}

Next, let $g_1\in L_2(\T)$ be defined by the formula
\begin{equation}\label{g1}
g_1(z)=zg(z)=-i\,\dist(z,K)^{-1/3}.
\end{equation}
Then using (\ref{h}) and (\ref{g1}), we obtain
$$
\langle h_z,g_1\rangle=\!\!\int_\T h_z(w)\overline{g_1(w)}\,\mu(dw)=
i\!\!\int_\T\frac{\psi(w)\,\dist(w,K)^{1/3}\!\!}{z-w}\dist(w,K)^{-1/3}\,\mu(dw)=
i\!\!\int_\T\frac{\psi(w)}{z-w}\,\mu(dw).
$$
Applying Lemma~\ref{conj1}, we see that $\langle
h_z,g_1\rangle=i(\psi_0)_-(z)$, where $\psi_0(z)=z^{-1}\psi(z)$.
Using (\ref{gamma}), we have $(\psi_0)_-(z)=z^{-1}\gamma(z^{-1})$.
By (\ref{ii}), $\gamma(z^{-1})=-i$ and therefore $\langle
h_z,g_1\rangle=iz^{-1}\psi_-(z)=i(-i)z^{-1}=z^{-1}$. Thus
\begin{equation}\label{p2}
\langle h_z,g_1\rangle=z^{-1}\ \ \ \text{for each $z\in K$.}
\end{equation}

Finally, from (\ref{h}) and (\ref{g1}) it follows that
$$
\langle h,g_1\rangle=\!\!\int_\T h(w)\overline{g_1(w)}\,\mu(dw)=
i\!\!\int_\T\psi(w)\,\mu(dw)=i\langle \psi,{\bf 1}\rangle,
$$
where ${\bf 1}$ is the constant $1$ function. On the other hand,
looking at the shape (\ref{gamma}) of the Fourier series of $\psi$,
we immediately see that $\langle \psi,{\bf 1}\rangle=0$. Hence
$\langle h,g_1\rangle=0$, which together with (\ref{p1}) and
(\ref{p2}) implies (\ref{m02}). The proof of Lemma~\ref{main0} is
complete and so is the proof of Theorem~\ref{main}.

\section{Concluding remarks \label{s5}}

If we replace finite rank perturbations by compact perturbations,
Question~S becomes relatively easy. Namely, an operator of the shape
$I+K$ can be hypercyclic \cite{sal2}, where $K$ is a compact
operator on a separable infinite dimensional Hilbert space.
Moreover, $K$ may be chosen to be nuclear. On the other hand, an
operator of the shape $I+K$ with $K$ being of finite rank, can not
be cyclic.

Theorem~\ref{main} naturally gives rise to the following question.

\begin{question}\label{q1} Does there exist a hypercyclic rank $1$
perturbation of a unitary operator?
\end{question}

It is worth noting that the above proof of Theorem~\ref{main}
provides a hypercyclic rank 1 perturbation of a unitary operator if
we can construct $K$, $h$ and $g$ as in Lemma~\ref{main0} with the
additional property that $h\in\spannn\{h_\lambda:\lambda\in K\}$.
This additional requirement seems to be difficult to achieve.

Recall that a bounded linear operator $T$ on a Banach space $X$ is
called {\it mixing} if for any two non-empty open sets $U,V\subseteq
X$, $T^n(U)\cap V\neq \varnothing$ for all sufficiently large
$n\in\N$. Equivalently $T$ is mixing if and only if for any infinite
set $A\subset \N$, there exists $x=x(A)\in X$ such that $\{T^nx:x\in
A\}$ is dense in $X$. Thus mixing condition is a strong form of
hypercyclicity. The following question seems to be natural and
interesting.

\begin{question}\label{q2} Does there exist a mixing finite rank
perturbation of a hyponormal operator?
\end{question}

\bigskip {\bf Acknowledgements.} \ The author would like to thank the
referee for helpful comments. 

\small\rm

\vskip1truecm

\scshape

\noindent  Stanislav Shkarin

\noindent Queens's University Belfast

\noindent Department of Pure Mathematics

\noindent University road, Belfast, BT7 1NN, UK

\noindent E-mail address: \ \ \ {\tt s.shkarin@qub.ac.uk}


\begin{thebibliography}{99}

\itemsep=-2pt

\bibitem{bg}F.~Bayart and S.~Grivaux, \it Hypercyclicity and
unimodular point spectrum, \rm J. Funct. Anal. \bf 226\rm\ (2005),
281--300

\bibitem{bama}F.~Bayart and E.~Matheron, \it Dynamics of  linear
operators, \rm Cambridge University Press, Cambridge, 2009

\bibitem{bm}F.~Bayart and E.~Matheron, \it Hyponormal
operators, weighted shifts and weak forms of supercyclicity, \rm
Proc. Eninb. Math. Soc. \bf49\rm\ (2006), 1--15

\bibitem{belov}A.~Belov, \it On the Salem and Zygmund problem with respect
to the smoothness of an analytic function that generates a Peano
curve, \rm Math. USSR-Sb. \bf 70\rm\ (1991), 485--497

\bibitem{bdon}P.~Bourdon, \it Orbits of hyponormal operators, \rm
Michigan Math. J. \bf 44\rm\ (1997), 345--353

\bibitem{moga}E.~Gallardo-Guti\'errez and
A.~Montes-Rodr\'{\i}guez, \it The role of the angle in supercyclic
behavior, \rm J. Funct. Anal. \bf 203\rm\ (2003), 27--43

\bibitem{hw}H.~Hilden and L.~Wallen, \it Some cyclic
and non-cyclic vectors of certain operators, \rm Indiana Univ. Math.
J. \bf23\rm\ (1973), 557--565

\bibitem{ansa}C.~Kitai, \it Invariant closed sets for linear
operators, \rm Thesis, University of Toronto, 1982

\bibitem{sal1}H.~Salas, \it Supercyclicity and weighted shifts, \rm Studia
Math. \bf135\rm\ (1999),  55--74

\bibitem{sal2}H.~Salas, \it Hypercyclic weighted shifts, \rm
Trans. Amer. Math. Soc. \bf347\rm\ (1995), 993--1004

\bibitem{z}A.~Zygmund, \it Trigonometric series, \rm Vol.~I, Cambridge University Press,
Cambridge, 1988

\end{thebibliography}
\end{document}